\documentclass[a4paper,oneside, reqno,12pt]{amsart}
\usepackage[utf8]{inputenc}

\usepackage[top=3cm,bottom=3cm,left=3cm,right=2.5cm]{geometry} 
\usepackage[linkcolor = red, colorlinks = true]{hyperref}
\allowdisplaybreaks
\usepackage[dvips]{epsfig}
\usepackage{color}
\usepackage{amsgen, amstext,amsbsy,amsopn,amssymb, amsthm,mathrsfs,mathptmx,amsfonts,amssymb,amscd,amsmath,euscript,enumerate,verbatim,calc,enumitem,mathtools}
\usepackage{url}
\usepackage{hyperref}
\usepackage{multicol}
\usepackage{ragged2e}

\usepackage[normalem]{ulem}

\usepackage{soul,xcolor}

\theoremstyle{plain}
\newtheorem{theorem}{Theorem}[section]

\newtheorem{lemma}[theorem]{Lemma}

\theoremstyle{definition}
\newtheorem{definition}[theorem]{Definition}

\theoremstyle{remark}

\begin{document}
	
	\title[]{Equality of DSER elementary orthogonal group and Eichler-Siegel-Dickson transvection group}
	\author{Gayathry Pradeep\textsuperscript{1}$^{\ast}$}\thanks{$^{\ast}$Corresponding author}
	\author{Ambily Ambattu Asokan\textsuperscript{2}}
	\author{Aparna Pradeep Vadakke Kovilakam\textsuperscript{3}}
	\address{\newline\textsuperscript{1,2}Department of Mathematics
		\newline Cochin University of Science and Technology
		\\
		Cochin University P.O.
		\newline Cochin 682022, Kerala
		\newline India
		\newline \textsuperscript{3} Kerala School of Mathematics
		\newline Kozhikode 673571, Kerala
		\newline India 
		\vspace{2mm}
		\newline {\textsuperscript{1}\it{Email: }\tt ambily@cusat.ac.in}
		\newline {\textsuperscript{2}\it{Email: }\tt aparnapradeepvk@gmail.com}
		\newline {\textsuperscript{3}\it{Email: }\tt gayathrypradeep221b@gmail.com}}

	\begin{abstract}
		We prove the Dickson-Siegel-Eichler-Roy (DSER) elementary orthogonal group, which was introduced by Amit Roy in 1968 and the Eichler-Siegel-Dickson transvection group, which is in literature in the works of Dickson, Siegel and Eichler, are equal over a commutative ring in which $2$ is invertible. We prove the equality in the free case by considering the odd and even case separately and then generalize this result by using the local-global principle. This result generalizes previous results concerning the equality of elementary orthogonal transvection groups.
		
	\end{abstract}

	
	\vspace{1mm}
	

	\date{\today}

	\maketitle
\begin{flushright}
	{\it Dedicated to Nikolai Vavilov}
	\end{flushright}
		
	\section{Introduction}
	The study of linear groups became prominent in mid $60$'s through the seminal papers of H. Bass and Bass-Milnor-Serre \cite{Bass1964, BMS}. Compared with linear and symplectic groups, the orthogonal groups do not have transvections when the characteristic of the field is not $2$. Instead, they have particular types of transformations in the role of transvections which are known as {\it Eichler-Siegel-Dickson transvections (or ESD Transvections)}. The ESD transvection over a vector space is defined as follows: Let $u$ be an isotropic vector and $v$ be a vector orthogonal to $u$ with $r=q(v)$, then the transformation $\sigma_{u,v}$ is defined by 
	\begin{equation*}
		\sigma_{u,v}(x) = x + u \langle v,x \rangle-v \langle u,x \rangle-ur \langle u,x \rangle.
	\end{equation*}
	
	\vspace{2mm}
	The first documentation of Eichler-Siegel-Dickson (ESD) transformations appeared in the work of L.E. Dickson in matrix form over finite fields in the 1901 edition of `Linear groups: With an exposition of the Galois field theory', which is republished as \cite{Dickson1958}. It reappeared on works of C.L. Siegel on indefinite quadratic forms  in the papers `\emph{\"Uber die analytische Theorie der quadratischen Formen II.}' and `\emph{\"Uber die {Z}etafunktionen indefiniter quadratischer {F}ormen. {II}}'. In \cite{Siegel1939}, Siegel used these transformations to define `\emph{Das Darstellungsmaß}' which means the measure of representation of $0$, by an indefinite quadratic form, which is equivalent to the representation number of $0$. In \cite{Eichler1952}, M. Eichler developed the theory by defining these transformations on vector spaces over general fields. In \cite{Wall1964}, C.T.C. Wall extended the definition of Eichler transformations to unimodular lattices. In \cite{Roy1968}, Amit Roy generalized these transformations to a non-degenerate quadratic space with a hyperbolic summand over a commutative ring. These transformations are called the {\it Dickson-Siegel-Eichler-Roy (DSER) elementary orthogonal transformations} and the group generated by these transformations is called {\it DSER elementary orthogonal group}.
	
	\vspace{2mm}
	
	For a non-degenerate quadratic space $Q$ and a finitely generated projective module $P$, H. Bass introduced the notion of the transvection group on $Q\!\perp\!\mathbb{H}(P)$ for studying the stability theorem for projective modules in \cite{Bass1973}. He formulated a case in which an ESD transvection can be expressed in terms of the maps from $Q \rightarrow P$ and $P^* \rightarrow Q$ (see \cite[Proposition 5.4]{Bass1973}). The generators of DSER elementary orthogonal group are of this type. The DSER elementary orthogonal transvection on $Q\!\perp\! \mathbb{H}(R)$ can be viewed as an `{\it orthogonal lift}' of the usual elementary transvections on $Q\oplus R$ (for more details see \cite[Lemma~1]{Tamagawa1958}). In this article we prove that any transvection generated by a unimodular vector $u$ and a vector $v$, which is orthogonal to $u$ can be expressed as product of the DSER elementary orthogonal transformations. The linear, symplectic, and orthogonal transvection groups generalize the classical linear, symplectic, and orthogonal groups respectively, as both groups coincide over free modules \cite{BakBasuRao2010}. Analogous to classical groups, A. Bak, R. Basu and R. A. Rao, H. Apte, P. Chattopadhyay proved several vital results like normality theorem, the local-global principle etc for the elementary subgroups of transvection groups and general quadratic groups in \cite{BakBasuRao2010, PratyushaApteRao2012, Basu2018}. In \cite{Pratyusha2019}, P. Chattopadhyay proved the equality of the isotropic transvection group and the elementary orthogonal transvection group. The study of odd hyperbolic unitary group over a form ring were done by V. A. Petrov and R. Preusser in \cite{Petrov2005, Preusser2020}. In \cite{AparnaAmbily2024}, A. A. Ambily and V. K. Aparna Pradeep compared the DSER elementary orthogonal group and the odd elementary hyperbolic unitary group and proved that they are conjugate subgroups of the general linear group.
	
	\vspace{2mm}
	
	Simultaneously, A. Bak generalized the ESD transvections to a non-degenerate quadratic space over a form ring $(R,\Lambda)$. Chapters 5,9 in \cite{HahnOMeara1989} discuss a detailed study of these transformations. Though these two transformation groups routed to a single transformation group on a vector space over a field, the comparison of these two groups appears to be complex. It's evident that the DSER elementary orthogonal group is a subset of the ESD transformation group, but showing the reverse was difficult. In this article, we compare these two groups in the free case with the elementary orthogonal group and obtain that these three groups coincide if we choose the form on quadratic space accordingly. In general, we also compare the relative cases in each group and prove the equality of the corresponding relative groups.
	
	\vspace{2mm}
	
	Moreover, automorphism groups play a significant role in number theory. Several cryptosystems are based on automorphism groups over commutative rings. The ESD transvection group generalizes the Chevalley group corresponding to root systems $B_l$ and $D_l$. Therefore, it is interesting to study the possible algorithoms based on this orthogonal subgroups as well.


	\section{Preliminaries}
	Let $R$ denote a commutative ring with unity in which $2$ is invertible.
	\vspace{1mm}
	
	For a natural number $n \in \mathbb N$, let ${\rm GL}_{n}(R)$ denote the group of all $n \times n$ invertible matrices with entries from $R$. Let $\varphi$ be an $n\times n$ invertible symmetric matrix. The {\it orthogonal group} ${\rm O}_{n}(R)$, with respect to the matrix $\varphi$ is defined to be 
	$${\rm O}_n(R)=\{\alpha\in {\rm GL}_{n}(R):\alpha^t\varphi \alpha=\varphi\}.$$
	For $r \in \mathbb N$, let $e_{i,j}$ denote the $2r \times 2r$ matrix with $(i,j)^{\mbox{th}}$ entry $1$ and all other entries $0$ and let $\widetilde{\psi}_r$ denote the symmetric matrix $\widetilde{\psi}_r:=\sum_{i=1}^{r}e_{2i-1,2i}+\sum_{i=1}^{r}e_{2i,2i-1}$. In this article, we consider the standard symmetric matrix $\tilde{\phi}_n$, defined by 
	$$\tilde{\phi}_n =\begin{cases}
		\widetilde{\psi}_r \mbox{ when } n=2r, \\
		2 \perp \widetilde{\psi}_r \mbox{ when } n=2r+1.
	\end{cases}$$
	
	Also we study the elementary subgroup of the orthogonal group. Odd elementary orthogonal group and even elementary orthogonal group are defined using different type of generators.
	
	\vspace{1mm}
	
	Let $\sigma$ denote the permutation of $n$ numbers given by $ \sigma(2i)=2i-1$ and $\sigma(2i-1)=2i$. Then for $1 \leq i \neq j \leq n$ and $z\in R$, the matrices $oe_{i,j}(z)$ defined by
	$$oe_{i,j}(z)=I_{2n}+e_{i,j}(z)-e_{\sigma(j),\sigma(i)}(z),$$ are orthogonal. The subgroup of ${\rm O}_{2r}(R)$ generated by the matrices of the form $oe_{i,j}(z)$ for $1\leq i,j \leq n$, $i\neq j,\sigma(j)$, is called the {\it even elementary orthogonal group} ${\rm EO}_{2r}(R)$.
	
	\vspace{1mm}
	In \cite{FC15}, Fasel-Calm\`{e}s introduced some elementary orthogonal transformations corresponding to the root system ${B}_l$ as follows:

	\begin{enumerate}
		\item $\lambda \mapsto Id +\lambda(e_{1,2i+1}-2e_{2i,1}-\lambda e_{2i,2i+1})$ 
		\item $\lambda \mapsto Id + \lambda(e_{1,2i} -2e_{2i+1,1} - \lambda e_{2i+1,2i})$
		\item $\lambda \mapsto Id + \lambda(e_{2i,2j} - e_{2j+1,2i+1})$
		\item $\lambda \mapsto Id + \lambda(e_{2i,2j+1} - e_{2j,2i+1})$
		\item $\lambda \mapsto Id+ \lambda(e_{2i+1,2j} - e_{2j+1,2i})$
	\end{enumerate}
	Let $F_{i}^1, F_{i}^2, F_{i,j}^3, F_{i,j}^4,F_{i,j}^5$ denote the matrices corresponding to these transformations.
	
	\vspace{1mm} 
	By direct computation, we can see that the Fasel matrices satisfy the following relations:
	$$F_{i,j}^3(z)=[F_{i}^2(z),F_{j}^2(1)],\,
	F_{i,j}^4= [F_{i}^1(z),F_{j}^1(1)] \mbox{ and }
	F_{i,j}^5= [F_{i}^1(z),F_{j}^2(1)].$$ 
	Thus the group generated by the matrices given in equations (1) to (5) is same as the group generated by the matrices $F_{1}^i(z)$ and $F_2^i(z)$. Also, we can easily verify that the matrices $F_{1}^i(z)$ and $F_2^i(z)$, for $1 \leq i \leq r,$ and $z\in R$ are orthogonal matrices.
	Define the {\it the odd elementary orthogonal group} as the subgroup of the orthogonal group generated by the matrices of the form $F_{1}^i(z)$ and $F_2^i(z)$, for $1 \leq i \leq r,$ and $z\in R$, and denote it by ${\rm EO}_{2r+1}(R)$. 
	
	\vspace{1mm}
	
	Now we recall the definition of relative version of the elementary orthogonal group. Let $I$ be an ideal of $R$. The {\it true relative group} of even rank ${\rm EO}_{2r}(I)$ is the group generated by $oe_{ij}(z)$ for $z\in I$, and $1\leq i,j \leq 2r$, $i\neq j, \sigma(j)$, and the {\it true relative group} of odd rank ${\rm EO}_{2r+1}(I)$ is the group generated by $F_i^1(z)$ and $F_i^2(z)$, for $z \in I$. For any $n \in \mathbb N$, the normal closure of ${\rm EO}_n(I)$ in ${\rm EO}_n(R)$ is called the {\it relative elementary orthogonal group}, denoted by ${\rm EO}_n(R,I)$.
	
	\vspace{2mm}
	{\bf Remark:} The notation of {\it true relative group} is proposed by J. Tits in \cite{JTits1976}, while the terminology of {\it true relative group} is established by B. Nica in \cite{BNica2015} for the elementary linear groups.
	\vspace{2mm}
	
	Now we consider the orthogonal group on a quadratic $R$-module. A quadratic $R$-module $(M,\langle \ ,\ \rangle)$ is a finitely generated projective $R$-module $M$ together with a quadratic form $q$ on it, where  $\langle \ ,\ \rangle$ is the bilinear form associated with the quadratic form $q$ on $M$. The bilinear form $\langle \ ,\ \rangle$ on $M$ is said to be {\it non-degenerate} if $\langle x,y \rangle =0$ for all $y \in M$ implies $x=0$. A quadratic R-module is called a quadratic space if the corresponding bilinear form is non-degenerate.
	
	\vspace{1mm}
	The orthogonal group ${\rm O}_{R}(M)$ corresponding to a quadratic $R$-module $M$ is defined as follows: $${\rm O}_R(M):=\{ \sigma \in {\rm End}(M) | \langle \sigma x, \sigma y \rangle=\langle x, y \rangle\},$$
	where ${\rm End}(M)$ denotes the set of all endomorphisms on $M$.
	\begin{definition}[{{\it Hyperbolic space}}]
		Let $P$ be a finitely generated projective $R$-module. Then the {\it hyperbolic space} on $P$, denoted by $\mathbb{H}(P)$, is the quadratic space $P\oplus P^*$ with the bilinear form defined as follows:
		$$\langle(x_1,f_1),(x_2,f_2)\rangle_h :=f_2(x_1)+f_1(x_2), \mbox{ for }x_1,x_2\in P \mbox{ and } f_1,f_2 \in P^*,$$ where $P^*$ is the dual space ${\rm Hom}_R(P,R)$ of $P$.
	\end{definition}
	
	\section{Eichler-Siegel-Dickson Transformations}
	In this section, we study the classical Eichler transvection on a quadratic $R$-space. A detailed account of these transformations appeared in the works of Dickson, Siegel, Eichler, and C.T.C Wall in \cite{Dickson1958,Siegel1939,Eichler1952,Wall1964}. Later, H Bass, in \cite{Bass1973}, generalised the theory and defined the unitary transformations to prove stability theorems on the general unitary modules. 
	
	\begin{definition}
		Consider a non-degenerate quadratic $R$-module $(M,\langle\ ,\ \rangle,q)$. Let $u,v \in M$ with $\langle u,v \rangle =0$, $q(u)=0$, $r=q(v)$ and $u$ is unimodular. Define the {\it Eichler-Siegel-Dickson (ESD) transvection} $\sigma_{u,v} \in {\rm End}(M)$ by
		\begin{equation*}
			\sigma_{u,v}(x) = x + u \langle v,x \rangle-v \langle u,x \rangle-ur \langle u,x \rangle.
		\end{equation*}
		
		\noindent The group generated by these transvections is called the {\it Eichler-Siegel-Dickson transvection group} on $M$ and is denoted by ${\rm TransO} (M,\langle\ ,\ \rangle)$.
	\end{definition}

	{\bf Remark:} Let $V$ be a vector space over a skew field (division ring). Then a linear transformation $T:V \rightarrow V$ is 
	called a {\it transvection} if ${\rm rank}(T-I)=1$ and ${\rm Im}(T-I)\subseteq {\rm Ker}(T-I)$. 
	
	For an ESD transvection $\sigma_{u,v}$ we have the following properties:
	
	\begin{enumerate}
		\item $\sigma_{u,v}(u) = u + u \langle v,u \rangle-v \langle u,u \rangle-ur \langle u,u \rangle=u$.
		\item $\sigma_{u,v}(v) = v + u \langle v,v \rangle-v \langle u,v \rangle-ur \langle u,v \rangle=v+2ru$.
	\end{enumerate}
	
	Thus, \begin{equation*}
		\begin{split}
			(\sigma_{u,v}-I)(x) &= \langle v,x \rangle u- \langle u,x \rangle v -r \langle u,x \rangle u\\
			\sigma_{u,v}((\sigma_{u,v}-I)(x))&= \langle v,x \rangle \sigma_{u,v}(u)- \langle u,x \rangle\sigma_{u,v}(v)- r\langle u,x \rangle\sigma_{u,v}(u)\\
			&= \langle v,x \rangle u - \langle u,x \rangle(v+2ru)- r\langle u,x \rangle u\\
			&=(\sigma_{u,v}-I)(x)- 2ru\langle u,x\rangle.
		\end{split}
	\end{equation*}
	
	Consequently, ESD transvection is not a $transvection$. But we follow the terminology as used in \cite{Bass1973}. In \cite{HahnOMeara1989}, this map is mostly referred to as {\it Eichler transformations}.
	
	\vspace{2mm}
	
	In this paper, we compare the transformation group ${\rm TransO} (M,\langle\ ,\ \rangle)$ with the Dickson-Siegel-Eichler-Roy (DSER) elementary orthogonal group introduced by Amit Roy in \cite{Roy1968}. Amit Roy used the following transformations to prove the cancellation theorem for quadratic modules.
	
	\begin{definition}
		Let $(Q,\langle\ ,\ \rangle)$ be a quadratic $R$-space over $R$, and $M=Q \perp \mathbb{H}(P)$ for some finitely generated projective $R$-module $P$. Then there is an isomorphism $d_{\langle \ ,\ \rangle}:Q \to Q^* $, defined as $ d_{\langle\ ,\ \rangle}(z)(w) = \langle z,w\rangle$, for $z,w \in Q$.
		
		\vspace{1mm}
		
		Let $\alpha:Q \to P$ be an $R$-linear map. Corresponding to $\alpha, $ consider the transpose map $\alpha^t:P^* \to Q^*$ defined by $\alpha^t(\phi_p)=\phi_p\circ\alpha $, for $\phi_p \in P^*$. Now let $\alpha^*:P^* \to Q $ be defined as $\alpha^*=d_{\langle\ ,\ \rangle}^{-1}\circ \alpha^t$. Then the orthogonal transformation $E_{\alpha} $ on $M$ is defined as:
		$$E_{\alpha}(z,x,f)= (z-\alpha^*(f), x+\alpha(z)-\frac{1}{2}\alpha \alpha^*(f),f),~ \mbox{ for } ~z \in Q,~ x \in P \mbox{ and } f\in P^*.$$
		Similarly for an $R$-linear map $\beta:Q \to P^*$, there exists $\beta^t : P^{**} \to Q^*$ defined by $\beta^t(\phi_p^*)=\phi_p^* \circ \beta$ for $\phi_p^*\in P^{**}$. Now define $\beta^*:P \to Q $ as $\beta^*=d_{\langle \ ,\ \rangle}^{-1}\circ \beta^t \circ i$, where $i $ is the natural isomorphism from $P \to P^{**}$. Then the orthogonal transformation $E_{\beta}^* $ on $M$ is defined as:
		$$E^*_{\beta}(z,x,f)= (z-\beta^*(x),x,f+\beta(z)-\frac{1}{2}\beta\beta^*(x)), ~\mbox{ for } ~z \in Q, ~x \in P \mbox{ and } f\in P^*.$$
		
		The orthogonal transformations $E_{\alpha}$ and $E^{*}_{\beta}$ are called {\it DSER elementary orthogonal transformations}. The group ${\rm EO}_{R}(Q,\mathbb{H}(P))$ generated by these orthogonal transformations is called the {\it DSER elementary orthogonal group}. 
	\end{definition}
	
	\begin{definition}[{\bf Elementary orthogonal transvections}]
		Let $Q$ be a non-degenerate quadratic $R$-space. Take $w \in Q$ and $M = Q\perp \mathbb{H}(R)^{m}$. Let $\{x_{i}, f_{i}\}_{i=1}^{m}$ be the basis for $\mathbb{H}(R)^{m}$. For $(z,r,s) \in M$, define  two maps  $(E_{1i}^w)$ and $(E_{2i}^w)$ on $M$ as 
		
		$$E_{1i}^w: (z,r,s) \mapsto (z- \langle s,x_{i} \rangle w,\ r+\langle z,w \rangle x_{i} - \langle s,x_{i} \rangle q(w) x_{i},\ s),$$
		$$E_{2i}^w: (z,r,s) \mapsto (z- \langle r,f_{i} \rangle w,\ r, \ s+\langle z,w \rangle f_{i} - \langle r,f_{i} \rangle q(w) f_{i}).$$
	\end{definition}
	One can verify that the maps $E_{1i}^w$ and $E_{2i}^w$ are orthogonal transvections on $M$ and the group generated by these transvections is called ${\rm ETransO}(M,\langle , \rangle)$.
	
	In particular, for $M = Q\perp \mathbb{H}(R)$, consider the basis $\{x,f\}$ for $\mathbb{H}(R)$. For $(z,r,s) \in M$, define the orthogonal transvections $E_{1}^w$ and $E_{2}^w$ as  
	$$E_1^w: (z,r,s) \mapsto (z- \langle s,x\rangle w,\ r+\langle z,w \rangle x - \langle s, x \rangle q(w)x,\ s),$$
	$$E_2^w: (z,r,s) \mapsto (z- \langle r, f\rangle w,\ r, \ s+\langle z,w \rangle f - \langle r,f \rangle q(w)f).$$

	\section{Comparison of DSER elementary orthogonal transformation group and the elementary orthogonal transvection group}
	In this section we prove the equality of the DSER elementary orthogonal group and the elementary transvection group when the underlying projective module $P$ is free.
	First we prove the result for a free module of rank one and then extend the result to an arbitrary free module.
	\begin{lemma}\label{DSER=ETrans1}
		The DSER elementary orthogonal group ${\rm EO}_{R}(Q,\mathbb{H}(R))$ on $M= Q\perp \mathbb{H}(R) $ where $Q$ is a quadratic space over $R$ coincides with $ {\rm ETransO}({M,\langle,\rangle })$.
	\end{lemma}
	\begin{proof}
		
		In \cite{Roy1968}, Amit Roy defined the DSER transformation $E_\alpha$ and $E_{\beta} ^ *$ on $Q \perp \mathbb{H}(P)$ in the particular case when $P=Rx$, is a free $R$-module of rank $1$.
		
		Let $f \in P^*$ with $f(x)=1$, so that $\{x,f\}$ is an $R$-basis of $\mathbb{H}(P)$. For $w \in Q$, consider the linear map $ \alpha : Q \to Rx $, which can be expressed in the form,
		$$\alpha(z)=\langle z,w\rangle x,$$ for $z \in Q$. 
		
		\noindent Since $w$ represents $\alpha$, we can denote $E_\alpha$ by $E_w$ and can be expressed as 
		\begin{align*}
			&E_w(z) = z+ \langle z,w \rangle x, \mbox{ for } z \in Q, \\
			&E_w(x) = x, \mbox{ and } \\
			&E_w(f) = -w-q(w)x+f.
		\end{align*}
		
		\noindent Similarly the orthogonal transformation $E^*_\beta$ for some $\beta : Q \to R^*=R $, such that $\beta(z)=\langle z,w\rangle f$ can be expressed as;
		\begin{align*}
			&E^*_w(z) = z+ \langle z,w \rangle f, \mbox{ for } z \in Q, \\
			&E^*_w(x) = -w+x-q(w)f, \mbox{ and } \\
			&E^*_w(f) = f.
		\end{align*}
		
		Here if we take $P=R$, then for some $(r,s) \in \mathbb{H}(R)= R^2$ such that $r=r'x $ and $s=s'f$ with $r's'=1$ and $z \in Q$, we have
		\begin{align*}
			& E_w(z) = z+ \langle z,w \rangle x,\\
			& E_w(r) = r'E_w(x)= r'x =r,\\
			& E_w(s) = s'E_w(f) = -s'w-s'q(w)x+s=-\langle s,x \rangle w -\langle s,x \rangle q(w)x+s.
		\end{align*}
		
		Therefore we can observe that 
		\begin{align*}
			&E_w(z,r,s)=(z- \langle s,x\rangle w,\ r+\langle z,w \rangle x - \langle s, x \rangle q(w)x,\ s)=E_1^w(z,r,s)\\
			&E_w^*(z,r,s)=(z-\langle r, f\rangle w,r,s + \langle z,w \rangle -rq(w))=E_2^w(z,r,s),
		\end{align*}
		
		for all $z \in Q$ and $(r,s) \in \mathbb{H}(R)$. Hence we get $E_w=E_1^w \text{ and } E_w^*= E_2^w .$
	\end{proof}
	
	The {\it splitting lemma} for the DSER elementary orthogonal group is proved in {\cite{Suresh1994}} as follows:
	
	\begin{lemma}[{\rm \cite[Lemma~1.2]{Suresh1994}}]\label{splitting}
		Let $\alpha_1,\alpha_2\in {\rm Hom }(Q,P) ( {\rm or }\; \beta_1,\beta_2\in {\rm Hom}(Q,P^*))$, then we have the splitting for the elementary transformations $E_{\alpha}$ $($and $E_{\beta}^*)$ as follows

		$$E_{\alpha_1+\alpha_2}=E_{\frac{\alpha_1}{2}} E_{\alpha_2} E_{\frac{\alpha_1}{2}}({\rm or \;} E_{\beta_1+\beta_2}^*=E_{\frac{\beta_1}{2}}^* E_{\beta_2}^* E_{\frac{\beta_1}{2}}^* ).$$
	\end{lemma}
	Now, we can prove the equality of the elementary transvection group and the DSER elementary orthogonal group ${\rm EO}_{R}(Q,\mathbb{H}(P)$, where $P$ is a free module of rank $m$. 
	\begin{lemma}\label{DSER=ETrans}
		For a quadratic space $Q$ over $R$, The DSER elementary orthogonal group ${\rm EO}_{R}(Q,\mathbb{H}(R)^m)$ on $M= Q\perp \mathbb{H}(R)^m$ coincides with $ {\rm ETransO}({M,\langle\,,\,\rangle })$.
	\end{lemma}
	
	\begin{proof}
		Define $\alpha_i:Q \rightarrow P$ by $$\alpha_i(z)=\langle z, w \rangle x_i.$$
		Then clearly, we have $E_{1i}^w=E_{\alpha_i} \in {\rm EO_{R}(Q,\mathbb{H}(R)^m)}$ for $i=1,2, \ldots, m$. Similarly we can prove that $E_{2i}^w=E_{\beta_i} \in {\rm EO_{R}(Q,\mathbb{H}(R)^m)}$ by choosing $\beta_i:Q \rightarrow P^*$ to be $$\beta_i(z)=\langle z, w \rangle f_i.$$ Thus ${\rm ETransO}(M,\langle\,,\,\rangle) \subseteq {\rm EO}_{R}(Q,\mathbb{H}(R)^m)$.
		
		\vspace{2mm}
		
		Now for any $\alpha:Q \rightarrow P$, we have $\alpha(q)=\sum \alpha(q)_ix_i$, where $\alpha(q)_i$ is the coefficient of $x_i$ in the representation of $\alpha(q)$ in $P$. Define $$\alpha_i(q)=\alpha(q)_ix_i.$$ By Theorem \ref{DSER=ETrans1}, we have $E_\frac{\alpha_i}{2} \in {\rm ETransO}(M,\langle\,,\,\rangle)$ and applying the fact that $\alpha =\sum \alpha_i$ in Theorem \ref{splitting}, we get ${\rm E_\alpha}\in {\rm ETransO}(M,\langle\,,\,\rangle)$. Similarly for $\beta:Q\rightarrow P^*$, we have $E_{\frac{\beta_i}{2}} \in {\rm ETransO}(M,\langle\,,\,\rangle)$, which implies ${\rm EO}_{R}(Q,\mathbb{H}(R)^m) \subseteq {\rm ETransO}(M,\langle\,,\,\rangle)$.
	\end{proof}

	\section{Relative subgroups}
	Similar to the definition of relative subgroups of classical groups with respect to an ideal $I$ of $R$, we have the notion of relative subgroups of the DSER elementary orthogonal group and the transvection group. The relative transvection groups are defined as follows:
	
	Let $I$ be an ideal of $R$. 
	
	\begin{definition}{\it The relative ESD transformation} on a finitely generated projective $R$-module $M$ with respect to the ideal $I$ is a transformation of the form
		\begin{equation*}
			\sigma_{u,v}(x) = x + u \langle v,x \rangle - v \langle u,x \rangle -ur \langle u,x \rangle,
		\end{equation*}
		for $u \in M$, $v \in IM $. The normal closure of the group generated by the relative ESD transvection in ${\rm TransO}(M,\langle\ ,\ \rangle)$ is called {\it the relative ESD transvection group} and is denoted by ${\rm TransO}(M,IM,\langle\ ,\ \rangle)$. 
	\end{definition}
	
	Now we define the relative DSER elementary orthogonal groups and relative elementary transvection groups as follows:
	
	\begin{definition}[{{\it Relative DSER Elementary transformations}}]
		The DSER elementary transformations $E_{\alpha}$ and $E_\beta^*$ for $\alpha :Q \rightarrow IP $ and $\beta: Q \rightarrow IP^*$ are called the relative DSER elementary orthogonal transformations. The group generated by relative DSER transformations is called  {\it the true relative DSER elementary orthogonal group} and is denoted by ${\rm EO}_{I}(Q ,\mathbb{H}(P))$. The normal closure of ${\rm EO}_{I}(Q ,\mathbb{H}(P))$ in ${\rm EO}_{R}(Q,\mathbb{H}(P))$ is the {\it relative DSER elementary orthogonal group} and is denoted by ${\rm EO}_{(R,I)}(Q ,\mathbb{H}(P))$.
	\end{definition}
	
	\begin{definition}[{\it Relative elementary orthogonal transvections}] 
		
		For $w \in IQ$ and $\{x_{i}, f_{i}\}_{i=1}^{m}$ be the basis for $\mathbb{H}(R)^{m}$, the group generated by the orthogonal transvections on $M$ as $(E_{1i}^w)$ and $(E_{2i}^w)$is called the {\it true relative elementary orthogonal transvections group} and is denoted by ${\rm ETransO}(IM).$ The normal subgroup generated by these transformations is called the {\it relative elementary orthogonal transvection group} and is denoted by ${\rm ETransO}(M,IM)$.
	\end{definition}
	
	\vspace{2mm}
	
	{\bf Remark:} In particular, if $w \in IQ,$ then $(w,0,0) \in IM$. Thus, the relative DSER elementary transformations are relative ESD transformations. i.e.,
	${\rm EO}_{(R,I)}(Q ,\mathbb{H}(R)^m)$ is a subgroup of ${\rm TransO}(M,IM,\langle \ ,\ \rangle)$. In Lemma \ref{DSER=ETrans}, we have proved the equality of elementary transvection group and the DSER elementary orthogonal group. This equality holds for the relative case as well. If we take $w\in IQ$, we get the corresponding $\alpha,\beta\in {\rm Hom}(Q,IR^m)$ and vice versa. 
	
	\vspace{2mm} 
	
	In this paper, we consider the converse of this subset relation. We first consider the case when $Q$ is a free module of rank $n,$ and compare the group ${\rm EO}_{(R,I)}(Q,\mathbb{H}(R)^m)$ and ${\rm Trans}_{\rm O}(M,IM)$ with the usual elementary orthogonal group ${\rm EO}_{n+2m}(R,I)$.

	\section{DSER elementary orthogonal transformations in free case}
	
	Let $Q $ be a free module of rank $n$ with a bilinear form $\langle \ ,\ \rangle$ on it and P be a free module of rank $m$. Then $P^*$ will be a free module of rank $m$. Let $\{z_1,z_2,\ldots,z_n\}$ be a basis of $Q,\; \{x_1,x_2,\ldots, x_m\}$ be a basis of $P$ and $\{f_1,f_2,\ldots, f_m\}$ be the corresponding dual basis of $P^*$.
	
	\vspace{2mm}

	For $1\leq i \leq m $ and $1\leq j \leq n $, let $\alpha_{i,j}:Q \mapsto P$ be defined as $\alpha_{i,j}(z_j)=x_i$ and $\alpha_{i,j}(z_k)=0$, for all $k \neq j$. Similarly for $1\leq i \leq m $ and $1\leq j \leq n $, define $\beta_{i,j}:Q \mapsto P^*$ by $\beta_{i,j}(z_j)=f_i$ and $\beta_{i,j}(z_k)=0$, for all $k \neq j$.
	
	\vspace{2mm}
	
	Then $\{ \alpha_{i,j}\}$ generates the group ${\rm Hom}(Q,P)$ and $\{\beta_{i,j}\}$ generates the group ${\rm Hom}(Q,P^*)$. i.e., for every $\alpha \in {\rm Hom}(Q,P)$ and $\beta \in {\rm Hom}(Q,P^*)$ there exists $ c_{i,j},d_{i,j} \in R $ such that $\alpha = \sum_{i,j}c_{i,j}\alpha_{i,j}$ and $\beta = \sum_{i,j}d_{i,j}\alpha_{i,j}$. By splitting property of DSER elementary orthogonal transformations in \ref{splitting}, it can be verified that $ \{E_{r\alpha_{i,j}},E^*_{r\beta_{i,j}}\}_{r\in R} $ generates ${\rm EO}_{R}(Q,\mathbb{H}(R))$. Now, let us consider the matrices corresponding to the DSER elementary orthogonal transformations.
	
	\vspace{2mm}
	
	The matrix corresponding to $r\alpha_{i,j}=$ matrix corresponding to $r\beta_{i,j}= re_{i,j}$ and the matrix corresponding to the bilinear form on $Q$ is $M_{\langle \ ,\ \rangle}=(M_{ij})=(\langle z_i,z_j \rangle).$
	
	\vspace{2mm}

	Let $D=(d_{i,j})=M_{\langle \ ,\ \rangle}^{-1}$. Then the matrix corresponding to $$\alpha_{i,j}^* = D \circ \alpha_{i,j}^t =\begin{pmatrix}
		d_{11}&d_{12}&\ldots&d_{1j}&\ldots&d_{1n}\\
		d_{21}&d_{22}&\ldots&d_{2j}&\ldots&d_{2n}\\
		\vdots&\vdots&\ddots&\vdots&\ddots&\vdots\\
		d_{n1}&d_{n2}&\ldots&d_{nj}&\ldots&d_{nn}\\
	\end{pmatrix}{\alpha}_{i,j}^t = \begin{pmatrix}
		0&0&\ldots&d_{1j}&\ldots&0\\
		0&0&\ldots&d_{2j}&\ldots&0\\
		\vdots&\vdots&\ddots&\vdots&\ddots&\vdots\\
		0&0&\ldots&d_{nj}&\ldots&0\\
	\end{pmatrix}$$ (is an $n\times m $ matrix with $j^{th}$ column of $D$ as $i^{th}$ column).
	
	\vspace{2mm}
	
	\noindent  Now by definition we have,
	\begin{align*}
		&E_{r\alpha_{i,j}}(z_k)= z_k+r\alpha_{i,j}(z_k)= z_k+rx_i\delta_{j,k}, \mbox{ for } z_k \in Q,\\
		&E_{r\alpha_{i,j}}(x_k)= x_k, \mbox{ for } x_k \in P,\\
		&E_{r\alpha_{i,j}}(f_k)= -r\alpha_{i,j}^*(f_k)-\frac{r^2}{2}\alpha_{i,j} \alpha_{i,j}^*(f_k)+f_k=-r\sum_{l=1}^{n}d_{lj}z_l \delta_{i,k}-\frac{r^2}{2}d_{jj}x_{i} \delta_{i,k}+f_k, \mbox{ for } f_k \in P^*.
	\end{align*}
	
	\vspace{2mm}
	
	\noindent Then the matrix corresponding to $E_{r\alpha_{i,j}}$ is given by
	
	\vspace{2mm}
	
	$$\begin{pmatrix}
		I_{n\times n} &{\bf 0} &-r\alpha^*_{i,j}\\
		r\alpha_{i,j}&I_{m\times m}&-\frac{r^2}{2}\alpha_{i,j}\alpha^*_{i,j}\\
		{\bf 0}&{\bf 0}&I_{m \times m}\\
	\end{pmatrix}.$$
	
	\vspace{2mm}
	
	\noindent  Similarly for $E_{\beta_{i,j}}$, we have
	\begin{align*}
		&E^*_{r\beta_{i,j}}(z_k)= z_k+r\beta_{i,j}(z_k)=z_k+rf_i\delta_{j,k}, \mbox{ for } z_k \in Q,\\
		&E^*_{r\beta_{i,j}}(x_k)= -r\beta_{i,j}^*(x_k)+x_k-\frac{r^2}{2}\beta_{i,j}\beta_{i,j}^*(x_k)=-r\sum_{l=1}^{n}d_{lj}z_l \delta_{i,k}+x_k-\frac{r^2}{2}d_{jj}f_{i} \delta_{i,k}, \mbox{ for } x_k \in P,\\
		&E^*_{r\beta_{i,j}}(f_k)= f_k, \mbox{ for } f_k \in P^*.
	\end{align*}
	
	\noindent Then the matrix corresponding to $E_{r\beta_{i,j}}$ is given by
	$$\begin{pmatrix}
		I_{n\times n} &-r\beta^*_{i,j}&{\bf 0}\\
		{\bf 0}&I_{m \times m}&{\bf 0}\\
		r\beta_{i,j}&-\frac{r^2}{2}\beta_{i,j}\beta^*_{i,j}&I_{m\times m}\\
	\end{pmatrix}.$$
	
	\vspace{2mm}
	
	{\bf Remark:} In the subsequent sections, we shall consider the module $\mathbb{H}(R)^m$ by a change of basis to $\mathbb{H}(R^m)$.

	\section{Comparison of Relative DSER group and relative elementary orthogonal group}
	
	In this section, we prove the equality of DSER elementary orthogonal group ${\rm EO}_{R}(Q,\mathbb{H}(R))$ for a free quadratic $R$-space $Q$ of rank $n$ with the elementary orthogonal group ${\rm EO}_{n+2}(R)$. This section is divided into two parts, considering both the even case and odd case separately.
	First we prove Lemma \ref{EOeqEOStd} , which is a particular case of Lemma 2.7 in \cite{AmbilyRao2020}.
	
	\begin{lemma}\label{EOeqEOStd}
		If $(Q,q)$ be a free quadratic space of rank $n=2r$ for $r\geq 1$ and let $ \widetilde{\phi}_n$ be the matrix corresponding to $q$ with respect to the standard basis. Then $${\rm EO}_{R}(Q ,\mathbb{H}(R)) = {\rm EO}_{n+2}(R).$$
	\end{lemma}
	
	\begin{proof}
		
		Let $M=Q\perp\mathbb{H}(R)$. We have already proved that ${\rm EO}_{R}(Q ,\mathbb{H}(R)) = {\rm ETransO}(Q\perp \mathbb{H}(R))$. Now we compare ${\rm ETransO}(Q\perp \mathbb{H}(R))$ with ${\rm EO}_{n+2}(R)$. 
		
		Let $\mathcal{B}=\{e_i\}_{i=1, \ldots, 2r}$ be the basis of $Q$ and choose $w=\sum_{i=1}^{n} w_ie_i \in Q$ for $w_i\in R$. Let $(x,f)\in \mathbb{H}(R)$ such that $f(x)=1$. Then the matrix corresponding to $E_1^w $ with respect to the basis $\{e_1,e_2,\dots,e_{2r},x,f\}$ of $Q\perp\mathbb{H}(R)$ is

		\begin{equation}\label{eq1}
			\begin{pmatrix}
				1 & 0 & 0 & \dots & 0 & 0 & -w_1 \\
				0 & 1 & 0 & \dots & 0 & 0 & -w_2 \\
				\vdots & \vdots & \vdots & \ldots & \vdots & \vdots & \vdots\\
				0 & 0 & 0 & \dots & 1 & 0 & -w_n \\
				& & w \widetilde{\phi}_n & & & 1 &-q(w) \\
				0&0&0&\dots & 0 & 0 & 1\\  \end{pmatrix} = \textstyle{\prod_{i=1}^{n-1}} oe_{i,n+2}(\frac{-w_i}{2})  \cdot oe_{n,n+2}(-w_n) \cdot \textstyle{\prod_{i=1}^{n-1}} oe_{n-i,n+2}(\frac{-w_{n-i}}{2}).\end{equation}

		\vspace{3mm}
		
		\noindent Similarly the matrix corresponding to $E_2^w$ is
		
		\vspace{2mm}

		\begin{equation}\label{eq2}
			\begin{pmatrix}
				1 & 0 & 0 & \dots & 0 & -w_1 & 0 \\
				0 & 1 & 0 & \dots& 0 & -w_2 & 0 \\
				\vdots & \vdots & \vdots & \ldots & \vdots & \vdots & \vdots\\
				0 & 0 & 0 & \dots & 1 & -w_n & 0 \\
				0 & 0 & 0 & \dots & 0 & 1 & 0\\
				& & w \widetilde{\phi}_n & & &-q(w) & 1 \\\end{pmatrix} =\textstyle{\prod_{i=1}^{n-1}} oe_{i,n+1}(\frac{-w_i}{2})\cdot  oe_{n,n+1}(-w_n) \cdot \textstyle{\prod_{i=1}^{n-1}} oe_{n-i,n+1}(\frac{-w_{n-i}}{2}).\end{equation}
		
		\vspace{2mm}
		
		\noindent Therefore, $E_1^{w}, {E_{2}^{w}} \in {\rm EO}_{n+2}(R)$. Since $E_1^{w}$ and $ {E_2^{w}} $ generates the group ${\rm ETransO}(Q\perp \mathbb{H}(R))$, we have ${\rm ETransO}(Q\perp \mathbb{H}(R)) \subseteq {\rm EO}_{n+2}(R)$.
		
		\vspace{2mm}

		To prove the other inclusion, consider the relation $oe_{i,j}(ab)=[oe_{i,k}(a),oe_{k,j}(b)]$. That is, we get $oe_{i,j}(x)=[oe_{i,n+2}(x),oe_{n+2,j}(1)]$. Also we have, $ oe_{i,n+2}(x) = E^{-xe_i}_1$ and $oe_{n+2,j}(1) = oe_{\sigma(j),n+1}(-1) = E^{e_{\sigma(j)}}_{2}$. Therefore $ oe_{i,j}(x)= [E^{-xe_i}_1, E^{e_{\sigma(j)}}_2] \in {\rm ETransO}(Q\perp\mathbb{H}(R))$, which means ${\rm EO}_{n+2}(R) \subseteq {\rm ETransO}(Q\perp\mathbb{H}(R))$.  
	\end{proof}	
	
	{Now we prove the relative version of the equality in Lemma \ref{EOeqEOStd} in a more general setting.
		\begin{lemma}\label{EOeqEOpsi}
			If $(Q,q)$ be a free quadratic space of rank $n=2r$ for $r\geq 1$ and let $ \widetilde{\phi}_n$ be the matrix corresponding to $q$ with respect to some basis $\mathcal{B}$. Then $${\rm EO}_{(R,I)}(Q ,\mathbb{H}(R)) = {\rm EO}_{n+2}(R,I).$$
		\end{lemma}
		
		\begin{proof}
			Define $M=Q\perp\mathbb{H}(R)$. From Eq. \ref{eq1} and Eq.\ref{eq2},  it is clear that if $w \in IQ$ then $w_i \in I$ for all $i$. Thus, ${\rm ETransO}(M,IM) \subseteq {\rm EO}_{n+2}(R,I)$.
			
			\vspace{2mm}
			
			We have $oe_{i,j}(ab)=[oe_{i,k}(a),oe_{k,j}(b)]$. Thus for $x\in I, \; oe_{i,j}(x)=[oe_{i,n+2}(x),oe_{n+2,j}(1)]$. But $ oe_{i,n+2}(x) = E^{xe_i}_1$ and $oe_{n+2,j}(1) = oe_{\sigma(j),n+1}(-1) = E^{-1e_{\sigma(j)}}_2$. Hence, for $x\in I$, we obtain the commutator relation $ oe_{i,j}(x)= [E^{xe_i}_1, E^{-1e_{\sigma(j)}}_2] $. Thus, we get ${\rm EO}_{n+2}(I) \subseteq {\rm ETransO}(M,IM)$. That is,
			$$ {\rm EO}_{n+2}(I) \subseteq {\rm ETransO}(M,IM)={\rm EO}_{(R,I)}(Q,\mathbb{H}(R)) \subseteq {\rm EO}_{R}(Q,\mathbb{H}(R)) = {\rm EO}_{n+2}(R).$$
			Therefore, as a consequence of the normality of ${\rm EO}_{(R,I)}(Q,\mathbb{H}(R)) $ in ${\rm EO}_{(R)}(Q,\mathbb{H}(R))$, we obtain ${\rm EO}_{n+2}(R,I) \subseteq {\rm EO}_{(R,I)}(Q,\mathbb{H}(R)) $.
		\end{proof}
		
		\vspace{2mm}
		Now we consider the odd elementary orthogonal groups and the DSER elementary orthogonal group of odd rank.
		\vspace{2mm}
		
		\begin{lemma}\label{dser-eo-equal}
			If $Q=R $ with the quadratic form $q(z)=z^2$, then $${\rm EO}_R(Q, \mathbb{H}(R)^m) = {\rm EO}_{2m+1}(R).$$
		\end{lemma}
		\begin{proof}
			
			Define $\alpha_i:Q \to P$, as $\alpha_i(1)=x_i$ and $\beta_i:Q \to P^*$ as $\beta_i(1)=f_i$,  for $1\leq i\leq m$. Then,  for $\lambda \in R$, the generators of ${\rm EO}_R(Q, \mathbb{H}(R)^m) $ and $ {\rm EO}_{2m+1}(R)$ satisfy the following equations: $$F_i^1(\lambda)= E_{\alpha_{i}} (-2\lambda),$$ $$F_i^2(\lambda) = E^*_{\beta_{i}} (-2\lambda).$$ 
			As a consequence, we get ${\rm EO_{2m+1}(R)} = {\rm EO_{R}(Q, \mathbb{H}(R)^m)}$. 
		\end{proof}
		\noindent Now, we have analogous result for relative elementary orthogonal groups.
		\begin{lemma}\label{dser=eo_odd1}
			If $Q=R $ with the quadratic form $q(z)=z^2$, then $${\rm EO}_{(R,I)}(Q, \mathbb{H}(R)^m) = {\rm EO}_{2m+1}(R,I).$$
		\end{lemma}
		\begin{proof}
			The proof follows from a similar argument as in Theorem \ref{dser-eo-equal} for $\lambda\in I$.
		\end{proof}

		\begin{lemma}[\cite{AmbilyRao2024} Lemma~3.6] \label{dser=eo_odd2}
			Let $R$ be a commutative ring in which $2$ is invertible. Let $(Q,q)$ be a diagonalizable quadratic $R$-space of rank $n\geq1$, then $${\rm EO}_{R}(Q\perp \mathbb{H}(R)^{m-1},\mathbb{H}(R)) = {\rm EO}_{R}(Q,\mathbb{H}(R)^m).$$
		\end{lemma}
		
		\vspace{2mm}
		Therefore, we can combine the above 3 results as follows.
		\vspace{2mm}
		
		\begin{lemma}\label{dser=eo_odd}
			If $(Q,q)$ is a quadratic space of rank $2m-1$ with the matrix corresponding to $q$ is $2 \perp \widetilde{\psi}_{m-1}$ with respect to some basis $\mathcal{B}$ of $Q$, then $$ {\rm EO}_{R,I}(Q,\mathbb{H}(R))={\rm EO}_{2m+1}(R,I).$$
		\end{lemma}
		\begin{proof}
			The proof follows from Theorem \ref{dser=eo_odd1} and Theorem \ref{dser=eo_odd2}. 
		\end{proof}
		
		\section{Comparison of ESD transvection group and elementary orthogonal group.}
		In this section, we compare the $ESD$ transvection group on a free quadratic $R$-space with the elementary orthogonal group.
		Let $u=(u_1,u_1',u_1'',u_2',u_2'',\cdots,u_n',u_n'') \in R^{2n+1}$ be an isotropic unimodular vector and $v=(v_1,v_1',v_1'',v_2',v_2'',\cdots,v_n',v_n'') \in R^{2n+1} $ with $\langle u,v \rangle =0$ and $r:=q(v)$. Then the matrix corresponding to $\sigma_{u,v}$ with the standard quadratic form is as follows:
		
		$$\left( \begin{smallmatrix}
			1-2ru_1^2&u_1v_1''-v_1u_1''-ru_1u_1''&u_1v_1'-v_1u_1'-ru_1u_1'&\cdots&u_1v_n''-v_1u_n''-ru_1u_n''&u_1v_n'-v_1u_n'-ru_1u_n'\\
			2(u_1'v_1-v_1'u_1-ru_1'u_1)&1+u_1'v_1''-v_1'u_1''-ru_1'u_1''&-ru_1'^2&\cdots&u_1'v_n''-v_1'u_n''-ru_1'u_n''&u_1'v_n'-v_1'u_n'-ru_1'u_n'\\
			2(u_1''v_1-v_1''u_1-ru_1''u_1)&ru_1''^2&1+u_1''v_1'-v_1''u_1'-ru_1''u_1'&\cdots&u_1''v_n''-v_1''u_n''-ru_1''u_n''&u_1''v_n'-v_1''u_n'-ru_1''u_n'\\
			\vdots& \vdots&\vdots&\ddots&\vdots&\vdots\\
			2(u_n'v_1-v_n'u_1-ru_n'u_1)&u_n'v_1''-v_n'u_1''-ru_n'u_1''&u_n'v_1''-v_n'u_1''-ru_n'u_1''&\cdots&1+u_n'v_n''-v_n'u_n''-ru_n'u_n''&1-ru_n'^2\\
			2(u_n''v_1-v_n''u_1-ru_n''u_1)&u_n''v_1''-v_n''u_1''-ru_n''u_1''&u_n''v_n'-v_n''u_1'-ru_n''u_1'&\cdots&-ru_n''^2&1+u_n''v_n'-v_n''u_n'-ru_n''u_n'\\
		\end{smallmatrix}\right).$$

		\vspace{4mm}
		
		Consider a quadratic $R$-space $M$ of rank $2n$. The following results are obtained by A. A. Suslin and V. I. Kope\u{\i}ko in \cite{SuslinKopeiko77}. These results are used in proving the stability theorems for orthogonal group.

		\begin{lemma}[\cite{SuslinKopeiko77}, Proposition 2.10]\label{TransEven} 
			Let $M$ be a free quadratic $R$-space of rank $2n$, for $n\geq 3$. Let $u,v \in M$ such that $q(u)=\langle u,v \rangle =0,$ and $u$ is unimodular. Then $\sigma_{u,v} \in {\rm EO}_{2n}(R)$.
		\end{lemma}
		
		Now we need to consider the case when the non-degenerate quadratic module $Q$ have odd rank. To begin with, we state some results on the odd elementary orthogonal group similar to the results proved for even elementary orthogonal groups in \cite{SuslinKopeiko77}.
		
		\begin{lemma}\label{oddEO_1}
			
			If $A =(a_{ij}) \in {\rm GL}_n(I)$, then 
			$$\tilde{A}=\begin{pmatrix}
				1&0&0&0&0&\ldots&0&0\\
				0&a_{11}&0&a_{12}&0& \ldots &a_{1n}&0\\
				0&0&(A^T)^{-1}_{11}&0&(A^T)^{-1}_{12}&\ldots &0&(A^T)^{-1}_{1n}\\
				0&a_{21}&0&a_{22}&0& \ldots &a_{2n}&0\\
				0&0&(A^T)^{-1}_{21}&0&(A^T)^{-1}_{22}&\ldots &0&(A^T)^{-1}_{2n}\\
				\vdots&\vdots&\vdots&\vdots& \ddots &\vdots&\vdots\\
				0&a_{n1}&0&a_{n2}&0& \ldots &a_{nn}&0\\
				0&0&(A^T)^{-1}_{n1}&0&(A^T)^{-1}_{n2}&\ldots &0&(A^T)^{-1}_{nn}\\
			\end{pmatrix} \in {\rm O}_{2n+1}(I).$$
			\vspace{2mm}
			In particular if $A\in {\rm E}_n(I)$, then the corresponding $\tilde{A} \in {\rm EO}_{2n+1}(I)$.
		\end{lemma}
		\begin{proof}Both the assertions are obvious using Lemma 2.1 in \cite{SuslinKopeiko77}.\end{proof}
		
		\begin{lemma}\label{oddEO_2}
			Let $A =(a_{ij}) \in {\rm Alt_n(I)}$. Then the following two types of matrices 
			\begin{multicols}{2}
				\begin{enumerate}[label=(\roman*)]
					\item $\begin{pmatrix}
						1&0&0&0&0&\ldots&0&0\\
						0&1 & a_{11} &0&a_{12}& \ldots &0&a_{1n}\\ 
						0&0&1&0&0&\ldots &0&0\\
						0&0 & a_{21} &1&a_{22}& \ldots &0&a_{2n}\\ 
						0&0&0&0&0&\ldots &0&0\\
						\vdots&\vdots&\vdots&\vdots&\vdots&\ddots&\vdots&\vdots\\
						0&0 & a_{n1} &0&a_{n2}& \ldots &1&a_{nn}\\ 0&0&0&0&0&\ldots &0&1\\
					\end{pmatrix} $
					\item $\begin{pmatrix}
						1&0&0&0&0&\ldots &0&0\\
						0&1&0&0&0&\ldots &0&0\\
						0&a_{11} &1&a_{12}& 0&\ldots &a_{1n}&0\\ 
						0&0&0&0&0&\ldots &0&0\\
						0&a_{21} &0&a_{22}&1& \ldots &a_{2n}&0\\ 
						\vdots&\vdots&\vdots&\vdots&\vdots&\ddots&\vdots&\vdots\\
						0&0&0&0&0&\ldots &1&0\\
						0&a_{n1} &0&a_{n2}& 0&\ldots &a_{nn}&1\\
					\end{pmatrix}$
				\end{enumerate}
			\end{multicols} 
			belong to ${\rm EO}_{2n+1}(I)$.  
		\end{lemma}

		\begin{proof}
			This Lemma follows from Lemma 2.2 in \cite{SuslinKopeiko77}.
		\end{proof}

		\begin{lemma}
			Suppose $n\geq 3$. Let $u=(0,u_1',u_1'',u_2',u_2'',\ldots,u_n',u_n'')\in R^{2n+1}$ be an isotropic unimodular vector and $v=(0,v_1',v_1'',v_2',v_2'',\ldots,v_n',v_n'') \in R^{2n+1} $ with $\langle u,v \rangle =0$. Then $\sigma_{u,v} \in {\rm EO}_{2n+1}(R)$.
		\end{lemma}
		
		\begin{proof}
			The matrix corresponding to $\sigma_{u,v}$ is
			$$\left(\begin{smallmatrix}
				1&0&0&\cdots&0&0\\
				0&1+u_1'v_1''-v_1'u_1''-ru_1'u_1''&-ru_1'^2&\cdots&u_1'v_n''-v_1'u_n''-ru_1'u_n''&u_1'v_n'-v_1'u_n'-ru_1'u_n'\\
				0&ru_1''^2&1+u_1''v_1'-v_1''u_1'-ru_1''u_1'&\cdots&u_1''v_n''-v_1''u_n''-ru_1''u_n''&u_1''v_n'-v_1''u_n'-ru_1''u_n'\\
				\vdots& \vdots&\vdots&\ddots&\vdots&\vdots\\
				0&u_n'v_1''-v_n'u_1''-ru_n'u_1''&u_n'v_1''-v_n'u_1''-ru_n'u_1''&\cdots&1+u_n'v_n''-v_n'u_n''-ru_n'u_n''&1-ru_n'^2\\
				0&u_n''v_1''-v_n''u_1''-ru_n''u_1''&u_n''v_n'-v_n''u_1'-ru_n''u_1'&\cdots&-ru_n''^2&1+u_n''v_n'-v_n''u_n'-ru_n''u_n'\\
			\end{smallmatrix}\right).$$
			
			Now by Lemma \ref{TransEven}, we have the corresponding matrix $\sigma_{u,v} \in 1\perp {\rm EO}_{2n}(R) \subseteq {\rm EO}_{2n+1}(R)$.	\end{proof}
		\begin{lemma}\label{oddEO0u'u''andv1v'v''}
			Suppose $n\geq 3$. Let $u=(0,u_1',u_1'',u_2',u_2'',\ldots,u_n',u_n'')\in R^{2n+1} $ be an isotropic unimodular vector and $v=(v_1,v_1',v_1'',v_2',v_2'',\ldots,v_n',v_n'') \in R^{2n+1} $ with $\langle u,v \rangle =0$ and $v_i'=0$ or $v_i''=0$ for all $i\in \{1,2,\ldots, n\}$. Then $\sigma_{u,v} \in {\rm EO}_{2n+1}(R)$.
		\end{lemma}
		
		\begin{proof}
			Without loss of generality, we suppose that $v_i'=0$ for all $i$. Then the matrix corresponding to $\sigma_{u,v}$ is 
			$$ \left(\begin{smallmatrix}
				1&-v_1u_1''&-v_1u_1'&\cdots&-v_1u_n''&-v_1u_n'\\
				2u_1'v_1&1+u_1'v_1''-ru_1'u_1''&-ru_1'^2&\cdots&u_1'v_n''-ru_1'u_n''&-ru_1'u_n'\\
				2u_1''v_1&-ru_1''^2&1-v_1''u_1'-ru_1''u_1'&\cdots&u_1''v_n''-v_1''u_n''-ru_1''u_n''&u_1''v_n'-v_1''u_n'-ru_1''u_n'\\
				\vdots& \vdots&\vdots&\ddots&\vdots&\vdots\\
				2u_n'v_1&u_n'v_1''-ru_n'u_1''&u_n'v_1''-ru_n'u_1''&\cdots&1+u_n'v_n''-ru_n'u_n''&-ru_n'^2\\
				2u_n''v_1&u_n''v_1''-v_n''u_1''-ru_n''u_1''&u_n''v_n'-v_n''u_1'-ru_n''u_1'&\cdots&-ru_n''^2&1-v_n''u_n'-ru_n''u_n'\\
			\end{smallmatrix}\right).$$
			\vspace{1mm}
			
			\noindent By direct computations, we can observe that $\sigma_{u,v}$ can be written as a product 
			$$\sigma_{u,v}=\sigma_1 \sigma_2 \sigma_1\sigma_3\sigma_4,$$ where the components  $\sigma_i$'s are given as follows.
			$$\sigma_1=\begin{pmatrix}
				1&-\frac{v_1}{2}u_1''&0&\cdots&-\frac{v_1}{2}u_n''&0\\    
				0&1&0&\cdots&0&0\\
				u_1''{v_1}&-\frac{v_1^2}{4}u_1''^2&1&\cdots&-\frac{v_1^2}{4}u_1''u_n''&0\\
				\vdots& \vdots&\vdots&\ddots&\vdots&\vdots\\
				0&0&0&\cdots&1&0\\
				u_n''{v_1}&-\frac{v_1^2}{4}u_n''u_1''&0&\cdots&-\frac{v_1^2}{4}u_n''^2&1\\
			\end{pmatrix},$$
			\vspace{2mm}
			$$\sigma_2= \begin{pmatrix}
				1&0&-v_1u_1'&\cdots&0&-v_1u_n'\\
				2u_1'v_1&1&-\frac{v_1'^2}{2}u_1'^2&\cdots&0&-\frac{v_1^2}{2}u_1'u_n'\\
				0&0&1&\cdots&0&0\\
				\vdots& \vdots&\vdots&\ddots&\vdots&\vdots\\
				2u_n'v_1&0&-\frac{v_1^2}{2}u_n'u_1'&\cdots&1&-\frac{v_1^2}{2}u_n'^2\\      
				0&0&0&\cdots&0&1\\
			\end{pmatrix}, $$
			
			\vspace{2mm}
			$$ \sigma_3=\begin{pmatrix}
				1&0&0&\cdots&0&0\\    
				0&1+u_1'v_1''&0&\cdots&u_1'v_n''&0\\      
				0&0&1-v_1''u_1'&\cdots&0&-v_1''u_n'\\       
				\vdots& \vdots&\vdots&\ddots&\vdots&\vdots\\
				0&u_n'v_1''&0&\cdots&1+u_n'v_n''&0\\   
				0&0&-v_n''u_1'&\cdots&0&1-v_n''u_n'\\
			\end{pmatrix}, $$
			\vspace{2mm}
			$$ \sigma_4= \begin{pmatrix}
				1&0&0&\cdots&0&0\\     
				0&1&0&\cdots&0&0\\      
				0&0&1&\cdots&u_1''v_n''-v_1''u_n''&0\\      
				\vdots& \vdots&\vdots&\ddots&\vdots&\vdots\\
				0&0&0&\cdots&1&0\\     
				0&u_n''v_1''-v_n''u_1''&0&\cdots&0&1\\
			\end{pmatrix} .$$
			
			\vspace{2mm}
			\noindent Here $\sigma_1,\sigma_2 \in {\rm EO}_{R}(R,\mathbb{H}(R)^n)={\rm EO}_{2n+1}(R)$ and $\sigma_{3},\sigma_{4} \in {\rm EO}_{2n+1}(R)$ using Lemma \ref{oddEO_1} and Lemma \ref{oddEO_2}. Thus we get $\sigma_{u,v}\in{\rm EO}_{2n+1}(R)$. 	\end{proof}
		
		\noindent Now we prove the following theorem for a semilocal ring $R$.
		\begin{theorem}\label{TransOinEO}
			Let $R$ be a semilocal ring. Suppose $n\geq 3$. Let $u,v \in R^{2n+1} $, where $u$ is an isotropic unimodular vector with $\langle u,v \rangle =0$. Then $\sigma_{u,v} \in {\rm EO}_{2n+1}(R)$.
		\end{theorem}
		
		\begin{proof}
			Since $M$ is a non-singular quadratic space and $u$ an isotropic unimodular vector of $M$, $u$ can be completed to a hyperbolic pair $(u,w)$ of $M$. Let $\mathcal{B}$ be an ordered basis of $M$ with $u,w \in \mathcal{B}$. Let $\mathcal{B}=\{z',u,w,x_1,y_1,\ldots,x_{n-1},y_{n-1}\}$, where $(x_i,y_i)$ forms a hyperbolic pair and $z'\in M$ with $q(z')=1$ and $\langle z',u\rangle=\langle z',w\rangle=\langle z',x_i\rangle=\langle z',y_i\rangle=0$. (The existence of such a vector $z'$ is guaranteed by the cancellation theorem for semilocal ring, Theorem 8.1 in \cite{Roy1968}). The choice of the basis affirm the invariance of the symmetric matrix corresponding to the bilinear form we have began with. Now the vector corresponding to $u$ with respect to the new basis $\mathcal{B}$ is $(0,1,0,\ldots,0)\in R^{2n+1}$. Let $v$ be any vector orthogonal to $u$. Then the vector corresponding to $v$ has the form $(r,t,0,r_1,s_1,\ldots,r_{n-1},s_{n-1})$. define $v_i=ve_{i,i}$ for all $i\neq 3$, where $e_{ij}$ is the usual elementary matrix with 1 in ${ij}^{th}$ position and all other entries 0.
			Now we have $$\sigma_{u,v}=\prod_{i\neq 3}\sigma_{u,v_i}.$$
			By Theorem \ref{oddEO0u'u''andv1v'v''}, we know that $\sigma_{u,v_i}\in {\rm EO}_{2n+1}{(R)}$. Thus we get $\sigma_{u,v} \in {\rm EO}_{2n+1}(R)$.  \end{proof}
		
		\vspace{2mm}
		
		{\bf Remark:} Let $\mathcal{B}_o$ and $\mathcal{B}_n$ be two ordered basis of $M$ and let $T_o$ and $T_n$ be the matrix of the transformation with respect to the basis $\mathcal{B}_o$ and $\mathcal{B}_n$ respectively. Then by basic linear algebra, we have $T_n=P^{-1}T_oP,$ where $P\in {\rm GL}_n$ is the matrix corresponding to the change of basis (Transition matrix). 
		%
		%
		Now let $\varphi$ be the matrix corresponding to the symmetric bilinear form on $M$. We choose the new matrix in such a way that the new basis does not change the matrix associated to the form. Therefore, we have $P^T\varphi P=\varphi.$
		Thus we get $P\in {\rm O}_{2n+1}(R)$ and by the normality of elementary orthogonal group in ${\rm O}_{2n+1}(R)$, we can see that the change of basis does not affect the matrix being elementary.
		
		\vspace{2mm}
		\noindent We have the realtive version of Theorem \ref{TransOinEO} as follows.
		\begin{theorem}\label{TransEO}
			Let $R$ be a semilocal ring and $I$ be an ideal of $R$. Suppose $n\geq 3$. Let $u\in R^{2n+1} $ and $v\in I^{2n+1}$ where $u$ is unimodular vector with $\langle u,v \rangle =0$. Then $\sigma_{u,v} \in {\rm EO}_{2n+1}(R,I)$.
		\end{theorem}
		\begin{proof}
			The proof proceeds in the same way as in Theorem~\ref{TransOinEO}.	\end{proof}
		
		\noindent Now we can deduce a more general theorem as follows:
		
		\begin{theorem}\label{equalityfreecase}
			Let $R$ be a commutative ring in which 2 is invertible and $I$ be an ideal of $R$. Let $Q$ is a free quadratic $R$-space and let $M=Q \perp \mathbb{H}(R)^m$, such that the form $\langle\,,\, \rangle $ corresponds to the standard symmetric matrix $\varphi$ with respect to the standard basis on $M$.
			Then $${\rm TransO}(M,IM, \langle \,,\, \rangle )={\rm ETransO}(M,IM, \langle \,,\, \rangle ).$$
		\end{theorem}
		
		\begin{proof}
			For the given quadratic $R$-space $M=Q\perp\mathbb{H}(R)^m$, we have the following relations from Lemma \ref{TransEO}, Theorem \ref{DSER=ETrans}, Lemma \ref{dser=eo_odd} and Lemma \ref{EOeqEOpsi}
			$${\rm TransO}(M,IM,\langle \,,\, \rangle)\subseteq{\rm EO}_{n+2m}(R,I)={\rm ETransO}(M,IM,\langle \,,\, \rangle)\subseteq{\rm TransO}(M,IM,\langle \,,\, \rangle).$$	\end{proof}
		
		\section{Comparison of ESD transvection group and Elementary transvection group in general case}
		
		In the previous section, we have proved the equality of the transvection group and the elementary transvection group in the relative case, when the quadratic form corresponding to $Q$ is the standard symmetric matrix. In this section, we consider a more general case.
		\begin{lemma}\label{transvection-general}
			Let $M$ be a free quadratic $R$-space of rank $n$. Let $\langle \,,\,\rangle_{\varphi'}$ and $\langle\, ,\, \rangle_{\varphi^*}$ be two quadratic forms on $M$, where $\varphi'$ and $\varphi^*$ are the matrices corresponding to the quadratic forms with respect to an ordered basis $\mathcal{B}$ of $M$. Suppose that $\varphi' = \epsilon^t {\varphi^*}\epsilon$, for some $\epsilon \in {\rm GL}_n(R)$. Then we have, 
			${\rm TransO}(M,\langle\,,\,\rangle_{\varphi'}) = \epsilon^{-1} {\rm TransO}(M,\langle\,,\,\rangle_{\varphi^*})\epsilon.$
		\end{lemma}
		
		\begin{proof}
			Given that M is a free quadratic $R$-space. The matrix corresponding to the ESD transformation $\sigma_{u,v}$ on $(M,\langle,\rangle_{\varphi'})$ for $ u, v \in M$ is given by,
			$$I_n+uv^t\varphi' -vu^t\varphi' -uru^t\varphi' = (I_n+uv^t\varphi' -vu^t\varphi') (I_n -uru^t\varphi'), $$ 
			since $\langle u,v\rangle_{\varphi'} = 0$ and $\langle u,u\rangle_{\varphi'} = 0$ with $\langle v,v\rangle_{\varphi'} = r $.\newline  
			
			Similarly, the matrix corresponding to the transformation on $\sigma_{u,v}$ for $ u, v \in M, $ on $(M,\langle,\rangle_{\varphi^*})$ can be obtained by replacing $\varphi'$ by $\varphi^*$. Now, for $\epsilon \in {\rm GL}_n(R)$, 
			\begin{equation*}
				\begin{split}
					\epsilon^{-1}(I_n+uv^t\varphi^* &-vu^t\varphi^* -uru^t\varphi^*)\epsilon = (I_n+\epsilon^{-1}uv^t\varphi^*\epsilon -\epsilon^{-1}vu^t\varphi^*\epsilon -\epsilon^{-1}uru^t\varphi^*\epsilon)\\ 
					&=(I_n+\epsilon^{-1}u v^t{\epsilon^{t}}^{-1} \epsilon^{t} \varphi^*\epsilon -\epsilon^{-1} vu^t{\epsilon^{t}}^{-1} \epsilon^{t}\varphi^*\epsilon -\epsilon^{-1}uru^t{\epsilon^{t}}^{-1} \epsilon^{t}\varphi^*\epsilon)\\
					&=(I_n+\tilde u \tilde v^t \varphi'-\tilde v \tilde u^t\varphi' -\tilde ur\tilde u^t\varphi'),
				\end{split}
			\end{equation*}
			where $\tilde v = \epsilon^{-1} v $ and $\tilde u = \epsilon^{-1} u $ and $\langle v ,u \rangle_{\varphi^{*}}= v^t\varphi^{*}u= v^t{\epsilon^{-1}}^t\varphi^{*} \epsilon^{-1}u =\tilde v^t\varphi \tilde u = \langle \tilde v ,\tilde u \rangle_{\varphi}. $ 	\end{proof}
		
		\begin{lemma}\label{ETrans-general1}
			Let $Q$ be a free quadratic $R$-space of rank $n$, and $M=Q \perp R^2$. Let $\langle \,,\,\rangle_\varphi'$ and $\langle \,,\, \rangle_{\varphi^*}$ be two quadratic forms on $Q$ where $\varphi'$ and $\varphi^*$ are the matrices corresponding to the quadratic forms with respect an ordered basis $\mathcal{B}$ of $M$. Suppose that $\varphi = \epsilon^t {\varphi^*}\epsilon,$ for some $\epsilon \in {\rm GL}_n(R)$. Then
			${\rm ETransO}(M,\langle\,,\,\rangle_{\varphi' \perp \psi_1}) = (\epsilon^{-1}\perp I_2) {\rm ETransO}(M,\langle\,,\,\rangle_{\varphi^*\perp \psi_1)}(\epsilon \perp I_2).$
		\end{lemma}
		
		\begin{proof}
			For the generators $(E_{1}^{w})_{\varphi^*}$ of ${\rm ETransO}(M,\langle,\rangle_{\varphi^*\perp \psi_1)}$, we have
			\begin{equation*}
				\begin{split}
					(\epsilon^{-1}\perp I_2)(E_{1}^{w})_{\varphi^*} (\epsilon \perp I_2) &=
					\begin{pmatrix}
						\epsilon^{-1}&0&0\\
						0&1&0\\
						0&0&1\\
					\end{pmatrix}\begin{pmatrix}
						I_n&0&-w\\ 
						w^t{\varphi^*}^t&1&-q(w)\\
						0&0&1
					\end{pmatrix}\begin{pmatrix}
						\epsilon &0&0\\
						0&1&0\\
						0&0&1\\
					\end{pmatrix}\\
					&=\begin{pmatrix}
						I_n&0&-\epsilon^{-1}w\\ 
						{w}^t{\varphi^*}^t\epsilon&1&-q(w)\\
						0&0&1
					\end{pmatrix}\\
					&=\begin{pmatrix}
						I_n&0&-\epsilon^{-1}w\\ 
						{(\epsilon^{-1}w)}^t{\varphi'}^t&1&-q(w)\\
						0&0&1
					\end{pmatrix}\\
					&=(E_{1}^{\epsilon^{-1}w})_{\varphi'}.
				\end{split}
			\end{equation*}

			\vspace{2mm}
			
			\noindent Similarly for $(E_2^{w})_{\varphi^*}\in{\rm ETransO}(M,\langle,\rangle_{\varphi^*\perp \psi_1})$, we get
			\begin{equation*}
				\begin{split}
					(\epsilon^{-1}\perp I_2)(E_2^{w})_{\varphi^*}(\epsilon \perp I_2) &=
					\begin{pmatrix}
						\epsilon^{-1}&0&0\\
						0&1&0\\
						0&0&1\\
					\end{pmatrix}\begin{pmatrix}
						I_n&-w&0\\ 
						0&1&0\\
						w^t{\varphi^*}^t&-q(w)&1
					\end{pmatrix}\begin{pmatrix}
						\epsilon &0&0\\
						0&1&0\\
						0&0&1\\
					\end{pmatrix}\\
					&=\begin{pmatrix}
						I_n&-\epsilon^{-1}w&0\\ 
						0&1&0\\
						w^t{\varphi^*}^t\epsilon&1&-q(w)
					\end{pmatrix}\\
					&=\begin{pmatrix}
						I_n&-\epsilon^{-1}w&0\\ 
						0&1&0\\
						{(\epsilon^{-1}w)}^t{\varphi'}^t&1&-q(w)
					\end{pmatrix}\\
					&= (E_2^{\epsilon^{-1}w})_{\varphi'}.
				\end{split}
			\end{equation*}
			
			\noindent Since $\epsilon^{-1}w \in Q$, we get $(E_1^{\epsilon^{-1}w})_{\varphi'} , (E_2^{\epsilon^{-1}w})_{\varphi'} \in {\rm ETransO}(M,\langle,\rangle_{\varphi' \perp \psi_1})$. That is, 
			$$(\epsilon^{-1}\perp I_2) {\rm ETransO}(M,\langle,\rangle_{\varphi^*\perp \psi_1)}(\epsilon \perp I_2) \subseteq {\rm ETransO}(M,\langle,\rangle_{\varphi' \perp \psi_1}).$$  
			
			\noindent Similarly for $(E_1^{w})_{\varphi'} , (E_2^{w})_{\varphi'} \in {\rm ETransO}(M,\langle,\rangle_{\varphi' \perp \psi_1})$, we have 
			\begin{align*}
				&(E_1^{w})_{\varphi'} = (\epsilon\perp I_2)(E_1^{\epsilon w})_{\varphi^*} (\epsilon^{-1} \perp I_2 ),\\
				&(E_2^{w})_{\varphi'}=(\epsilon\perp I_2)(E_2^{\epsilon w})_{\varphi^*} (\epsilon^{-1} \perp I_2).
			\end{align*}
			Since $\epsilon w \in Q$, we can conclude that
			$$ {\rm ETransO}(M,\langle\,,\,\rangle_{\varphi' \perp \psi_1}) \subseteq (\epsilon^{-1}\perp I_2) {\rm ETransO}(M,\langle\,,\,\rangle_{\varphi^*\perp \psi_1)}(\epsilon \perp I_2). $$
		\end{proof}
		\noindent Now we extend Lemma \ref{ETrans-general1} to $Q\perp\mathbb{H}(R)^m$ as follows:
		\begin{lemma}\label{Etrans-general}
			Let $Q$ be a free quadratic $R$-space of rank $n$, and $M=Q \perp \mathbb{H}(R)^m$. Let $\langle \,,\,\rangle_{\varphi'}$ and $\langle \,,\, \rangle_{\varphi^*}$ be two quadratic forms on $Q$ where $\varphi'$ and $\varphi^*$ are the matrices corresponding to the quadratic forms with respect an ordered basis $\mathcal{B}$ of $M$. Suppose that $\varphi' = \epsilon^t {\varphi^*}\epsilon,$ for some $\epsilon \in {\rm GL}_n(R)$. Then
			$${\rm EO}_{R}(Q_{\varphi'},\mathbb{H}(R)^m) = (\epsilon^{-1}\perp I_m) {\rm EO}_{R}(Q_{\varphi^*},\mathbb{H}(R)^m)(\epsilon \perp I_m).$$
		\end{lemma}
		\begin{proof} Let $\alpha:Q \rightarrow R^m$ and $\beta:Q\rightarrow R^m$ be two homomorphisms. Then for the generators $(E_{\alpha})_{\varphi^*}$ of $({\rm EO}(Q\perp \mathbb{H}(R)^m))\langle,\rangle_{\varphi^*\perp \psi_1)}$, we have
			
			\begin{equation*}
				\begin{split}
					(\epsilon^{-1}\perp I_m)(E_{\alpha})_{\varphi^*} (\epsilon \perp I_m) &=
					\begin{pmatrix}
						\epsilon^{-1}&0&0\\
						0&I_m&0\\
						0&0&I_m\\
					\end{pmatrix}\begin{pmatrix}
						I_n&0&-\alpha^*\\ 
						\alpha&I_m&-\frac{1}{2}\alpha\alpha^*\\
						0&0&I_m
					\end{pmatrix}\begin{pmatrix}
						\epsilon &0&0\\
						0&I_m&0\\
						0&0&I_m\\
					\end{pmatrix}\\
					&=\begin{pmatrix}
						I_n&0&-\epsilon^{-1}\alpha^*\\ 
						\alpha\epsilon&I_m&-\frac{1}{2}\alpha\alpha^*\\
						0&0&I_m
					\end{pmatrix}\\
					&=\begin{pmatrix}
						I_n&0&-(\alpha\epsilon)^*\\ 
						\alpha\epsilon&I_m&-\frac{1}{2}\alpha\alpha^*\\
						0&0&I_m
					\end{pmatrix} {\text{ (with respect to $\varphi'$ on $Q$ )}}\\
					&=(E_{\alpha\epsilon})_{\varphi'},
				\end{split}
			\end{equation*}
			since ${\varphi^*}^{-1}\circ (\alpha\circ\epsilon)^t={\varphi^*}^{-1}\epsilon^t\alpha^t=\epsilon^{-1}{\varphi'}^{-1}\circ \alpha^t=\epsilon^{-1}\alpha^*$. Similarly we have
			$$(\epsilon^{-1}\perp I_m)(E_{\beta})_{\varphi^*} (\epsilon \perp I_m)= (E_{\beta\epsilon})_{\varphi'}.$$
			
			\noindent  Therefore, 
			$ (\epsilon^{-1}\perp I_m) {\rm EO}_{R}(Q_{\varphi^*},\mathbb{H}(R)^m)(\epsilon \perp I_m)\subseteq{\rm EO}_{R}(Q_{\varphi'},\mathbb{H}(R)^m) .$
			
			\vspace{2mm}
			\noindent Furthermore, for the generators ${({\rm E}_{\alpha})_{\varphi'}}$ and ${({\rm E}_{\beta})_{\varphi'}}$ of ${\rm EO}_{R}(Q,\mathbb{H}(R)^m)$ we have the relations,
			$$(\epsilon\perp I_m)(E_{\alpha})_{\varphi'} (\epsilon^{-1} \perp I_m)= (E_{\alpha\epsilon^{-1}})_{\varphi^*},$$
			$$(\epsilon\perp I_m)(E_{\beta})_{\varphi'} (\epsilon^{-1} \perp I_m)= (E_{\beta\epsilon^{-1}})_{\varphi^*}.$$
			
			\noindent  Hence we get, 
			${\rm EO}_{R}(Q_{\varphi'},\mathbb{H}(R)^m) \subseteq (\epsilon^{-1}\perp I_m) {\rm EO}_{R}(Q_{\varphi^*},\mathbb{H}(R)^m)(\epsilon \perp I_m).$
		\end{proof}

		\noindent {\bf Remark:} In general, Lemma \ref{transvection-general} and Lemma \ref{Etrans-general} holds for the relative case as well. That is, for an ideal $I$ of $R$, we have the following equalities:
		$${\rm TransO}(M,IM,\langle\,,\,\rangle_{\varphi'}) = \epsilon^{-1} {\rm TransO}(M,IM,\langle\,,\,\rangle_{\varphi^*})\epsilon,$$
		$${\rm EO}_{(R,I)}(Q_{\varphi'},\mathbb{H}(R)^m) = (\epsilon^{-1}\perp I_m) {\rm EO}_{(R,I)}(Q_{\varphi^*},\mathbb{H}(R)^m)(\epsilon \perp I_m).$$
		
		\noindent This can be proved using the same techniques as above.
		
		\begin{lemma}\label{transvection-general-relative}
			Let $M$ be a free quadratic $R$-space of rank $n$ and let $I$ be an ideal of $R$. Let $\langle \,,\,\rangle_{\varphi'}$ and $\langle\, ,\, \rangle_{\varphi^*}$ be two quadratic forms on $M$ where $\varphi'$ and $\varphi^*$ are the matrices corresponding to the quadratic forms with respect a basis $\mathcal{B}$ of $M$. Suppose that $\varphi' = \epsilon^t {\varphi^*}\epsilon$, for some $\epsilon \in {\rm GL}_n(R)$. Then
			$${\rm TransO}(M,IM,\langle\,,\,\rangle_{\varphi'}) = \epsilon^{-1} {\rm TransO}(M,IM,\langle\,,\,\rangle_{\varphi^*})\epsilon.$$
		\end{lemma}
		\begin{lemma}\label{ETrans-general-relative}
			Let $Q$ be a free quadratic $R$-space of rank $n$, and let $I$ be an ideal of $R$. Define $M=Q \perp \mathbb{H}(R)^m$. Let $\langle \,,\,\rangle_{\varphi'}$ and $\langle \,,\, \rangle_{\varphi^*}$ be two quadratic forms on $Q$ where $\varphi'$ and $\varphi^*$ are the matrices corresponding to the quadratic forms with respect a basis $\mathcal{B}$ of $M$. Suppose that $\varphi' = \epsilon^t {\varphi^*}\epsilon,$ for some $\epsilon \in {\rm GL}_n(R)$. Then
			$${\rm EO}_{(R,I)}(Q_{\varphi'},\mathbb{H}(R)^m) = (\epsilon^{-1}\perp I_m) {\rm EO}_{(R,I)}(Q_{\varphi^*},\mathbb{H}(R)^m)(\epsilon \perp I_m).$$
		\end{lemma}
		Now we consider a lifting of an ESD transvection to prove that every transvection is homotopic to the identity map on $M$.
		\begin{lemma}
			Let $M$ be a quadratic $R$-space and $\alpha \in {\rm TransO}(M,\langle\,,\,\rangle)$. Then there exists $\beta(X) \in {\rm TransO}(M[X],\langle\,,\,\rangle)$ such that $\beta(1)=\alpha$ and $\beta(0)=Id$.
		\end{lemma}
		\begin{proof}
			Given $\alpha \in {\rm TransO}(M,\langle,\rangle)$. Therefore $\alpha $ is product of transformations of the form
			
			$$ \alpha =\prod \sigma_{u_i,v_i}, $$
			
			\noindent where $\sigma_{u_i,v_i}(x)=x+\langle u_i,x \rangle v_i-\langle v_i,x \rangle u_i - q(v_i) \langle u_i,x \rangle u_i$, for $u_i,v_i \in Q$ with $\langle u_i,v_i \rangle = 0 = q(u_i) $ and $u_i$ is unimodular. Now let
			\begin{equation*}
				\sigma_iX(x) := 
				x+\langle u_i,x \rangle v_iX-\langle v_iX,x \rangle u_i - q(v_i)X^2 \langle u_i,x \rangle u_i, 
			\end{equation*}
			for $v_iX \in P[X]$.
			One can define $\beta(X)$ as $\beta(X):=\prod \sigma_{u_i,v_iX}$, which satisfies the conditions $\beta(1)=\prod \sigma_{u_i,v_i} = \alpha$ and $\beta(0)=\prod Id = Id$.
		\end{proof} 
		
		\noindent The relative analogue of this result can be expressed as follows:
		
		\begin{lemma}\label{lifttoM[X]}
			Let $M$ be a quadratic $R$-space and $\alpha \in {\rm TransO}(M,IM,\langle\,,\,\rangle))$. Then there exists $\beta(X) \in {\rm TransO}(M[X],IM[X],\langle\,,\,\rangle)$ such that $\beta(1)=\alpha$ and $\beta(0)=Id$.
		\end{lemma}
		
		\begin{proof}
			Take $u_i \in Q$ and $v_i \in IQ$ in the previous lemma. 
		\end{proof}
		
		 In \cite{AmbilyRao2020}, A. A. Ambily and R. A. Rao established Quillen's analogue of the local-global principle for the DSER elementary orthogonal group. In \cite{GayathryAmbily}, we prove the relative version of this result as follows:
		
		\begin{lemma}[Relative local-global principle]\label{RLG}
			Let $Q$ be a quadratic $R$-space of rank $n\geq 1$ and $P$ be a finitely generated projective module of rank $m\geq 2$, and $M=Q\perp \mathbb{H}(P)$. Let $\alpha(X) \in {\rm O}_{R[X]}(M[X])$ be such that $\alpha(0)=Id$. If $\alpha_{\mathfrak{m}}(X) \in {\rm EO}_{(R_\mathfrak{m}[X],I_{\mathfrak{m}}[X])}(Q_{\mathfrak{m}}[X],\mathbb{H}(P_{\mathfrak{m}}[X]))$, for all maximal ideal $\mathfrak{m}$ of $R$, then $\alpha(X) \in {\rm EO}_{(R[X],I[X])}(Q[X],\mathbb{H}(P[X]))$.
			
		\end{lemma}
		
		Now we can prove the equality of the ESD transvection group and the DSER elementary orthogonal group as follows: 
		
		\begin{theorem}\label{equality-general}
			Let $R$ be a commutative ring in which 2 is invertible and $I$ be an ideal of $R$. Let $(Q,q)$ be a non-degenerate quadratic module and $P$ be a finitely generated projective module with hyperbolic rank $m\geq 2$. Let $M=Q \perp \mathbb{H}(P)$ with form $\langle\,,\,\rangle=q\perp \langle\,,\,\rangle_h $. Assume that for every maximal ideal $ \mathfrak{m}$ over $R$, the form $\langle\,,\, \rangle_\mathfrak{m}$ corresponds to the standard form $\tilde{\phi}_{n+2m}$. Then $${\rm TransO}(M,IM, \langle\, ,\, \rangle )={\rm EO}_{(R,I)}(Q,\mathbb{H}(P)).$$
			
		\end{theorem}
		\begin{proof}
			It is easy to see that ${\rm EO}_{(R,I)}(Q,\mathbb{H}(P)) \subseteq {\rm TransO}(M,IM, \langle \,,\, \rangle )$.
			To verify the reverse inclusion, consider $\alpha \in {\rm TransO}(M,IM, \langle\, ,\, \rangle)$. Then by Lemma \ref{lifttoM[X]}, there exists $\beta(X) \in {\rm TransO}(M[X],IM[X], \langle \,,\, \rangle)$ such that $\beta(1)=\alpha $ and $\beta(0)= Id$. By Theorem \ref{equalityfreecase}, we get $\beta(X)_{\mathfrak{m}} \in {\rm TransO}(M_{\mathfrak{m}}[X],IM_{\mathfrak{m}}[X], \langle \, ,\, \rangle_{\varphi}) ={\rm EO}_{n+2m}(R_{\mathfrak{m}},I_{\mathfrak{m}})= {\rm EO}_{(R_{\mathfrak{m}},I_{\mathfrak{m}})}(Q_{\mathfrak{m}}[X],\mathbb{H}(P_{\mathfrak{m}}[X])$,  for all maximal ideal $\mathfrak{m}$ of $R$. Then, by Lemma \ref{RLG}, we obtain $\beta(X) \in {\rm EO}_{(R,I)}(Q[X],\mathbb{H}(P[X])) $. Substituting for $X=1$, we get $\alpha = \beta(1) \in {\rm EO}_{(R,I)}(Q,\mathbb{H}(P))$, which implies
			$${\rm TransO}(M,IM, \langle\, ,\, \rangle )\subseteq {\rm EO}_{(R,I)}(Q,\mathbb{H}(P)).$$
		\end{proof}

		\begin{theorem}
			Let $R$ be a commutative ring in which 2 is invertible and $I$ be an ideal of $R$. Let $(Q,q)$ be a non-degenerate quadratic module and let $M=Q \perp \mathbb{H}(P)$ with form $\langle\,,
			\,\rangle=q\perp \langle\,,
			\,\rangle_h$ for a finitely generated projective module of $P$ with hyperbolic rank $m\geq 2$ . Assume that for every maximal ideal $ \mathfrak{m}$ over $R$, the form $\langle\,,\, \rangle_\mathfrak{m} $ corresponds to the standard form $\tilde{\phi}$ such that $\varphi' = \epsilon^t \tilde{\phi}\epsilon,$ for some $\epsilon \in {\rm GL}_{n+2m}(R)$. Then $${\rm TransO}(M,IM, \langle\, ,\, \rangle_{\varphi'} )= {\rm EO}_{(R,I)}(Q,\mathbb{H}(P))_{\varphi'}.$$
			
		\end{theorem}
		\begin{proof}
			By using Lemma \ref{transvection-general-relative}, Lemma \ref{ETrans-general-relative} and Lemma \ref{equality-general}, we get
			\begin{equation*}
				\begin{split}
					{\rm TransO}(M,IM, \langle\, ,\, \rangle_{\varphi'} )&= (\epsilon^{-1}\perp I_m) {\rm TransO}(M, IM \langle\,,\,\rangle_{\tilde{\phi}})(\epsilon\perp I_m)\\
					&=(\epsilon^{-1}\perp I_m) {\rm EO}_{(R,I)}(Q,\mathbb{H}(P))_{\tilde{\phi}}(\epsilon\perp I_m)\\
					&={\rm EO}_{(R,I)}(Q,\mathbb{H}(P))_{\varphi'}.
				\end{split}	\end{equation*}\end{proof}
		
		\noindent \textbf{Acknowledgment:}
		The first author would like to thank KSCSTE Young Scientist Award Scheme (2021-KSYSA-RG) for providing grant to support this work. She would also like to thank SERB for SURE grant (SUR/2022/004894) and RUSA (RUSA 2.0-T3A), Govt. of India. The second author would like to thank NBHM and KSCSTE for the postdoctoral fellowship at KSOM and the third author would like to acknowledge the support of the CSIR Fellowship (09/0239(13173)/2022-EMR-I). The authors would like to thank Prof. B. Sury for his valuable suggestions.
		\\ 
		\vspace{2mm}

		\bibliographystyle{amsplain}

	\end{document}